\documentclass{amsart}
\usepackage{amsmath}
\usepackage{amssymb}
\usepackage{latexsym}
\usepackage{amsbsy}
\usepackage[all]{xy}
\usepackage{amscd}
\usepackage[dvips]{graphicx}

\newtheorem{Prop}{Proposition}[section]
\newtheorem{Thm}[Prop]{Theorem}
\newtheorem{Lemma}[Prop]{Lemma}
\newtheorem{Cor}[Prop]{Corollary}

\newtheorem{Remark}[Prop]{Remark}

\newtheorem{Definition}[Prop]{Definition}

    \def\bbq{{\mathbb Q}} \def\bb1{{\mathbb 1}}

  \def\bbc{{\mathbb C}}

\def\ra{\rightarrow}

\def\uq2{U_q(\hat{sl}_2)}

\def\bb{{\bf b}}

\def\mc{{\mathcal{C}}}

\def\md{{\mathcal{D}}}

\def\ue{{\underline{e}}}

\begin{document}

\title{On positivity for generalized cluster variables of  affine quivers}

\author{Xueqing Chen, Ming Ding and Fan Xu}
\address{Department of Mathematics,
 University of Wisconsin-Whitewater\\
800 W.Main Street, Whitewater,WI.53190.USA}
\email{chenx@uww.edu (X.Chen)}
\address{School of Mathematical Sciences and LPMC,
Nankai University, Tianjin, P.R.China}
\email{m-ding04@mails.tsinghua.edu.cn (M.Ding)}
\address{Department of Mathematical Sciences\\
Tsinghua University\\
Beijing 100084, P.~R.~China} \email{fanxu@mail.tsinghua.edu.cn
(F.Xu)}

\thanks{Ming Ding was supported by NSF of China (No. 11301282) and Specialized Research Fund for the Doctoral Program of Higher Education (No. 20130031120004) and Fan Xu was supported by NSF of China (No. 11471177).}


\bigskip

\begin{abstract}
It has been proved in \cite{LS} that cluster variables in cluster algebras of  every skew-symmetric cluster algebra are positive. We prove  that any regular generalized cluster variable of an affine quiver is positive. As a corollary, we obtain that generalized cluster variables of affine quivers are positive and construct various positive bases. This generalizes the results in \cite{Dupont} and \cite{DXX}.
\end{abstract}
\maketitle


\section{Introduction}
Cluster algebras were introduced by S. Fomin and A. Zelevinsky
\cite{FZ} in order to develop a combinatorial approach to study
problems of total positivity and canonical bases in quantum groups.  A cluster algebra $\mathcal{A}$ is a
subring of the field $\mathbb{Q}(x_1,\cdots,x_{n})$ of rational
fractions in $n$ indeterminates, and defined via a set of generators
constructed recursively. These generators are called cluster
variables and are grouped into subsets of fixed finite cardinality
called clusters. Monomials in the variables belonging to the same
cluster are called cluster monomials. By the Laurent phenomenon
\cite{FZ}, it is well-known that $\mathcal{A}\subset
\bigcap_{\mathrm{c}}\mathbb{Z}[\mathrm{c}^{\pm 1}]$ where
$\mathrm{c}$ runs over the clusters in $\mathcal{A}$. An element
$x\in \mathcal{A}$ is called  positive if $x\in
\bigcap_{\mathrm{c}}\mathbb{N}[\mathrm{c}^{\pm 1}]$ where
$\mathrm{c}$ runs over the clusters in $\mathcal{A}$. It is
conjectured that cluster variables are always positive \cite{FZ}. It has been proved in \cite{LS} that the positivity conjecture holds for the class of skew-symmetric cluster algebras.

Various $\mathbb{Z}-$bases were constructed in the cluster
algebras $\mathcal{A}(Q)$ \cite{SZ,CZ,GLS,Dupont1,DXX,GC}. When $Q$ is an affine
quiver, these
bases can be expressed as a disjoint union of the set of cluster
monomials and a set of generalized cluster variables $X_M$ associated to some
non-rigid regular $kQ-$modules $M$.  Hence, studying  the positivity for
regular generalized cluster variables is helpful for us to construct
canonical bases of cluster algebras.

For  types $\widetilde{A}$ and $\widetilde{D}$, G. Dupont has proved
\cite[Corollary 5.5]{Dupont} that if $M$ is an indecomposable
regular module in an exceptional tube $\mathcal{R}$, then $X_{M}\in
\mathcal{A}(Q)\bigcap\mathbb{N}[\mathrm{c}^{\pm 1}]$ for any cluster
$\mathrm{c}$ which is $\emph{compatible}$ with $\mathcal{R}$ (see Definition \ref{9}).

In this paper, we focus on cluster algebras of  affine quivers, that is of type $\widetilde{A},\widetilde{D}$ or $\widetilde{E}$  and prove that the coefficients of Laurent expansions  in normalized Chebyshev polynomials of the generalized cluster variable associated to quasi-simple modules in homogeneous tubes are
positive integer (Proposition \ref{Cheby}). As  an application, we deduce the positivity in regular generalized cluster variables (Theorem \ref{5}) and obtain various positive  integral bases  in cluster algebras of  affine quivers (Corollary \ref{basis}).

\section{preliminary}

Let $Q$ be an acyclic quiver with vertex set $Q_0=\{1,2,\cdots, n\}$
and we denote by $\mathcal{A}(Q)$ the associated cluster algebra.
Let $\bbc$ be the complex number field and $A=\bbc Q$ be the path
algebra of $Q$ and we denote by $P_i(resp.\ I_i)$ the indecomposable
projective$(resp.\ injective)$ $\bbc Q$-module with the simple
top$(resp.\ socle)$ $S_i$ corresponding to $i\in Q_0$. A connected
component in the Auslander-Reiten quiver of $\bbc Q-$modules is
called regular if it doesn't contain any projective or injective
$\bbc Q-$module. A $\bbc Q-$module is called regular if all of its
indecomposable direct summands belong to regular components. Every
regular component is of the form $\mathbb{Z}\mathbb{A}_{\infty}/(p)$
for $p\geq 0.$ If $p\geq 1$, the regular component is called a tube
of rank $p.$ If $p>1$, the regular component is called an
exceptional tube. If $p=1$, the regular component is called a
homogeneous tube. If $p=0$, the regular component is called a sheet.

Let $\md^b(Q)$ be the bounded derived category of $\mathrm{mod} \bbc
Q$ with the shift functor $T$ and the AR-translation $\tau$. The
cluster category associated to $Q$ was introduced in \cite{BMRRT} in
the spirit of categorification of cluster algebras. It is the orbit
category $\mathcal{C}(Q):=\md^{b}(Q)/F$ with $F=[1]\circ\tau^{-1}$.
Let $\mathbb{Q}(x_1,\cdots,x_n)$ be a transcendental extension of
$\mathbb{Q}.$ The Caldero-Chapton map of an acyclic quiver $Q$ is
the map
$$X_?^Q: \mathrm{obj}(\mc(Q))\ra\bbq(x_1,\cdots,x_n)$$ defined in \cite{CC} by
the following rules:
            \begin{enumerate}
                \item if $M$ is an indecomposable $\bbc Q$-module, then
                    $$
                        X_M^Q = \sum_{\textbf e} \chi(\mathrm{Gr}_{\ue}(M)) \prod_{i \in Q_0} x_i^{-\left<\ue, s_i\right>-\left <s_i, \underline{\mathrm{dim}}M - \ue\right
                        >};
                    $$
                \item if $M=P_i[1]$ is the shift of the projective module associated to $i \in Q_0$, then $$X_M^Q=x_i;$$
                \item for any two objects $M,N$ of $\mathcal C_Q$,
                we have
                    $$X_{M \oplus N}^Q=X_M^QX_N^Q.$$
            \end{enumerate}
Here, we denote by $\left <-,-\right >$ the Euler form on $\bbc
Q$-modules and $Gr_{\ue}(M)$ is the $\ue$-Grassmannian of $M,$ i.e. the
variety of submodules of $M$ with dimension vector $\ue.$
$\chi(\mathrm{Gr}_{\ue}(M))$ denote its Euler-Poincar$\acute{e}$
characteristic. We note that the indecomposable $\bbc Q$-modules and
$P_i[1]$ for $i\in Q_0$ exhaust the indecomposable objects of the
cluster category $\mc(Q)$. For any object $M\in \mc(Q)$, $X_M^Q$
will be called the generalized cluster variable for $M$. If $M$ is a
regular $\bbc Q$-module, $X_M^Q$ will be called the regular
generalized cluster variable. In \cite{Palu1}, the author
generalized the Caldero-Chapoton map for any cluster-tilting object
$T$ which is called the cluster character associated to $T$.

We recall that the normalized Chebyshev polynomial of the first kind
is defined by:
$$F_{0}(x)=2,F_{1}(x)=x,$$
$$F_{n}(x)=xF_{n-1}(x)-F_{n-2}(x)\ \ for\ any\ n\geq 2.$$
It is known that $F_{n}$ is characterized by
$$F_{n}(t+t^{-1})=t^{n}+t^{-n}$$
Following the terminology of \cite{DXX,Dupont1}, we denote the
generalized cluster variables associated to indecomposable regular
modules with quasi-length $n$ in homogeneous tubes for affine type
by $X_{n\delta}$.

\begin{Definition}\cite{Dupont}\label{9}
(1)\ A cluster-tilting object $T$ and
a regular component $\mathcal{R}$ are called to be compatible if $\mathcal{R}$ does
not contain any indecomposable direct summand of  $T$ as a
quasi-simple module.

(2)\ Let $T$ be any cluster-tilting object in $\mathcal{C}(Q)$. Set $F_{T}=Hom_{\mathcal{C}(Q)}(T,-)$,
$B=End_{\mathcal{C}(Q)}(T)$ and $\mathrm{c}_{T}=\{c_i|i\in Q_0\}$.
For any object $M$ in $\mathcal{C}(Q)$ which is not in $add(T[1])$
and for any $\underline{e}\in K_{0}(mod-B),$ the $\underline{e}-$component
of $X_{M}^{T}$ is
$$X_{M}^{T}(\underline{e})=\chi(\mathrm{Gr}_{\underline{e}}(F_{T}M))\prod_{i\in Q_0}c_{i}^{\langle S_{i},\underline{e}\rangle_{a}-\langle S_{i},F_{T}M\rangle}$$
(see \cite{Palu1} for definitions of  $\langle-,-\rangle$ and $\langle-,-\rangle_{a}$).
The interior of $X_{M}^{T}$ as
$$int(X_{M}^{T})=X_{M}^{T}-(X_{M}^{T}(0)+X_{M}^{T}([F_{T}M])$$
\end{Definition}

The following result which was proved in \cite{Dupont} will be useful for us to prove the positivity.

\begin{Thm}\cite{Dupont}\label{partial}
Let $Q$ be a quiver of affine types with at least three vertices. Let T be a cluster-tilting object such that there exists an exceptional tube $\mathcal{T}$
 compatible with $T$. Assume that $int(X_{M}^{T})\in \mathbb{N}[c_{T}^{\pm 1}]$ for any quasi-simple module $M$ in
  $\mathcal{T}$. Then, for
any $n\geq 1$, the following holds: $$F_{n}(X_{\delta})\in
\mathcal{A}(Q)\cap \mathbb{N}[c_{T}^{\pm 1}].$$
\end{Thm}

In \cite{Dav}, the authors has solved the
positivity conjecture for  all cluster variables of all skew-symmetric
 quantum cluster algebras. In particular, we have the following result.
\begin{Thm}\cite{Dav}\label{3}
Let $Q$ be a quiver of affine types. Let $T$ be any cluster-tilting object in
$\mathcal{C}(Q)$ and $M$ be an indecomposable rigid object not in
$add(T[1])$. Then, for any $\mathrm{e}\in K_{0}(mod-B),$ we have
$$\chi(\mathrm{Gr}_{\mathrm{e}}(F_{T}M))\geq 0.$$
\end{Thm}

According to Theorem \ref{partial} and Theorem \ref{3}, we can deduce the following
corollary which are proved in  \cite{Dupont} for types $\widetilde{A}$ and
$\widetilde{D}$:
\begin{Cor}\label{1}
Let $Q$ be a quiver of affine types with at least three vertices. Let T be a cluster-tilting object such that there exists an exceptional tube $\mathcal{T}$
 compatible with $T$.  Then, for
any $n\geq 1$, the following holds: $$F_{n}(X_{\delta})\in
\mathcal{A}(Q)\cap \mathbb{N}[c_{T}^{\pm 1}].$$
\end{Cor}

\section{positivity for regular generalized cluster variables of  affine quivers}
In this section, we will prove the positivity for regular generalized cluster variables of  affine quivers.

\begin{Lemma}\label{lemma}
Let $Q$ be a quiver of affine  types and $T$ be a cluster-tilting object of $\mathcal{C}(Q)$ such that each exceptional regular component $\mathcal{R}$  contains at least an indecomposable direct summand of  $T$. Then there exists some $i$ such that
$T_i$ is not a regular module and satisfies that $dim_{\bbc}\
Ext_{\mathcal{C}(Q)}^{1}(T_i,M(\delta))=1$, where $M(\delta)$ is a regular module  with dimension vector $\delta$ in any homogeneous tube.
\end{Lemma}
\begin{proof}
It is known that the shift functor $[1]=\tau$
induces an equivalence of cluster categories and maps
cluster-tilting object to another cluster-tilting object. For types $\widetilde{A}$, such $T_i$ is always exists because of  $\delta=(1,\cdots,1)$. Now we consider the types
$\widetilde{D}$ and $\widetilde{E}$. In  these cases,  we can also find such $T_i$, since then the cluster-tilting object $T$
contains at least  an indecomposable direct summand in the exceptional
regular component of rank 2.
\end{proof}

For any object $M\bigoplus \oplus_{i}P_i[1]$ in the cluster category $\mathcal{C}(Q)$ where $M$ is a $\mathbb{C}Q-$module and $P_i$ is a projective $\mathbb{C}Q-$module associated to the vertex $i$,  recall that the dimension vector is defined by
$$\underline{dim} (M\bigoplus \oplus_{i}P_i[1])=\underline{dim} M-\Sigma_{i}\underline{dim} S_i,$$
where $S_i$ is a simple $\mathbb{C}Q-$module associated to the vertex $i$. We are now ready to prove the following  result.

\begin{Prop}\label{Cheby}
Let $Q$ be a quiver of affine  types, then
$$F_{n}(X_{\delta})\in \mathcal{A}(Q)\cap
\mathbb{N}[\mathrm{c}^{\pm 1}]$$ for any cluster $\mathrm{c}$.
\end{Prop}
\begin{proof}
When $Q$ is a Kronecker quiver, it has been  proved in \cite{SZ}. In the following, we always assume that the affine quiver
$Q$ with at least three vertices.

Let $T=\bigoplus_{i=1}^{n}T_i$ be any cluster-tilting object in
$\mathcal{C}(Q)$, we only need to prove that $F_{n}(X_{\delta})\in
\mathcal{A}(Q)\cap \mathbb{N}[X_{T_1}^{\pm 1},\cdots,X_{T_n}^{\pm
1}]$.

According to Corollary \ref{1}, we only need to prove it for such cluster-tilting object $T$ satisfying that each exceptional regular component $\mathcal{R}$  contains at least an indecomposable direct summand of  $T$.
 By Lemma \ref{lemma}, there exists some $i$ such that
$T_i$ is not a regular module and satisfies that $dim\
Ext_{\mathcal{C}(Q)}^{1}(T_i,M(\delta))=1$. Then by the one dimension cluster multiplication formulas in \cite{CK2006}, we get
$$X_{\delta}X_{T_i}=X_{E}+X_{E'}.$$
Remark that $E$ and $E'$ are both indecomposable rigid objects but not regular 
$\mathbb{C}Q-$modules. For convenience, we denote the rigid object in $\mathcal{C}(Q)$ whose dimension vector $\underline{e}$ by $M(\underline{e})$.
Thus we have $M(\underline{dim}T_i)=T_i$, as  they are both  rigid objects with the same dimension vectors.

We now prove the following claim by induction:
$$F_{n}(X_{\delta})X_{T_i}=X_{M(\underline{dim}T_i+n\delta)}+X_{M(\underline{dim}T_i-n\delta)}\ \ \text{for}\ \ \underline{dim}T_i>n\delta;$$
$$F_{n}(X_{\delta})X_{T_i}=X_{M(\underline{dim}T_i+n\delta)}+X_{M(n\delta-\underline{dim}T_i)[-1]}\ \ \text{for}\ \ \underline{dim}T_i<n\delta.$$
 Firstly suppose that $n=1$, by the above discussions, we have
 $$X_{\delta}X_{T_i}=X_{M(\underline{dim}T_i+\delta)}+X_{M(\underline{dim}T_i-\delta)}\ \ \text{for}\ \ \underline{dim}T_i>\delta;$$
$$X_{\delta}X_{T_i}=X_{M(\underline{dim}T_i+\delta)}+X_{M(\delta-\underline{dim}T_i)[-1]}\ \ \text{for}\ \ \underline{dim}T_i<\delta.$$

Secondly suppose that $n=2$, which need to be divided into the following two cases:

(1) If $\underline{dim}T_i>\delta$, we have
\begin{eqnarray}
\nonumber F_{2}(X_{\delta})X_{T_i} & = & (X_{\delta}^{2}-2)X_{T_i}\\
\nonumber  & = & X_{\delta}(X_{M(\underline{dim}T_i+\delta)}+X_{M(\underline{dim}T_i-\delta)})
-2X_{T_i}\\
\nonumber & = & X_{M(\underline{dim}T_i+2\delta)}+X_{T_i}+X_{\delta}X_{M(\underline{dim}T_i-\delta)}
-2X_{T_i}.
\end{eqnarray}

We compute $X_{\delta}X_{M(\underline{dim}T_i-\delta)},$ which can be solved in the following cases:

In the case that $\underline{dim}T_i>2\delta$, we have
$$X_{\delta}X_{M(\underline{dim}T_i-\delta)}=X_{T_i}+X_{M(\underline{dim}T_i-2\delta)},$$
thus $$F_{2}(X_{\delta})X_{T_i}=X_{M(\underline{dim}T_i+2\delta)}+X_{M(\underline{dim}T_i-2\delta)};$$

In the case that $\underline{dim}T_i<2\delta$, we have
$$X_{\delta}X_{M(\underline{dim}T_i-\delta)}=X_{T_i}+X_{M(2\delta-\underline{dim}T_i)[-1]},$$
thus $$F_{2}(X_{\delta})X_{T_i}=X_{M(\underline{dim}T_i+2\delta)}+X_{M(2\delta-\underline{dim}T_i)[-1]}.$$

(2) If $\underline{dim}T_i<\delta$, we have

\begin{eqnarray}
\nonumber F_{2}(X_{\delta})X_{T_i} & = & (X_{\delta}^{2}-2)X_{T_i}\\
\nonumber  & = & X_{\delta}(X_{M(\underline{dim}T_i+\delta)}+X_{M(\delta-\underline{dim}T_i)[-1]})
-2X_{T_i}\\
\nonumber & = & X_{M(\underline{dim}T_i+2\delta)}+X_{T_i}+X_{M(2\delta-\underline{dim}T_i)[-1]}+X_{T_i}
-2X_{T_i}\\
\nonumber & = & X_{M(\underline{dim}T_i+2\delta)}+X_{M(2\delta-\underline{dim}T_i)[-1]}.
\end{eqnarray}

Now suppose that the above equations  hold for all $k\leq n$, we need to prove  them  for $k=n+1,$ which can be divided into the following two cases:

(1) If  $\underline{dim}T_i>n\delta$, we have

\begin{eqnarray}
\nonumber F_{n+1}(X_{\delta})X_{T_i} & = & (X_{\delta}F_{n}(X_{\delta})-F_{n-1}(X_{\delta}))X_{T_i}\\
\nonumber  & = & X_{\delta}(X_{M(\underline{dim}T_i+n\delta)}+X_{M(\underline{dim}T_i-n\delta)})\\
\nonumber  & - &(X_{M(\underline{dim}T_i+(n-1)\delta)}+X_{M(\underline{dim}T_i-(n-1)\delta)})\\
\nonumber & = & X_{M(\underline{dim}T_i+(n+1)\delta)}+X_{M(\underline{dim}T_i+(n-1)\delta)}+X_{\delta}X_{M(\underline{dim}T_i-n\delta)}\\
\nonumber & - & (X_{M(\underline{dim}T_i+(n-1)\delta)}+X_{M(\underline{dim}T_i-(n-1)\delta)}).
\end{eqnarray}

We need to compute $X_{\delta}X_{M(\underline{dim}T_i-n\delta)}:$

In the case that  $\underline{dim}T_i>(n+1)\delta$, we have
$$X_{\delta}X_{M(\underline{dim}T_i-n\delta)}=X_{M(\underline{dim}T_i-(n-1)\delta)}+X_{M(\underline{dim}T_i-(n+1)\delta)},$$
Thus we have $$F_{n+1}(X_{\delta})X_{T_i}=X_{M(\underline{dim}T_i+(n+1)\delta)}+X_{M(\underline{dim}T_i-(n+1)\delta)};$$

In the case that  $\underline{dim}T_i<(n+1)\delta$, we have
$$X_{\delta}X_{M(\underline{dim}T_i-n\delta)}
=X_{M(\underline{dim}T_i-(n-1)\delta)}+X_{M((n+1)\delta-\underline{dim}T_i)[-1]},$$
thus $$F_{n+1}(X_{\delta})X_{T_i}=X_{M(\underline{dim}T_i+(n+1)\delta)}+X_{M((n+1)\delta-\underline{dim}T_i)[-1]}.$$

(2) If  $\underline{dim}T_i<n\delta$, we have

\begin{eqnarray}
\nonumber F_{n+1}(X_{\delta})X_{T_i} & = & (X_{\delta}F_{n}(X_{\delta})-F_{n-1}(X_{\delta}))X_{T_i}\\
\nonumber  & = & X_{\delta}(X_{M(\underline{dim}T_i+n\delta)}+X_{M(n\delta-\underline{dim}T_i)[-1]})-F_{n-1}(X_{\delta})X_{T_i}\\
\nonumber & = & X_{M(\underline{dim}T_i+(n+1)\delta)}+X_{M(\underline{dim}T_i+(n-1)\delta)}\\
\nonumber & + & X_{\delta}X_{M(n\delta-\underline{dim}T_i)[-1]}
-F_{n-1}(X_{\delta})X_{T_i}.
\end{eqnarray}

When $\underline{dim}T_i>(n-1)\delta$, we have
$$X_{\delta}X_{M(n\delta-\underline{dim}T_i)[-1]}=X_{M((n+1)\delta-\underline{dim}T_i)[-1]}
+X_{M(\underline{dim}T_i-(n-1)\delta)},$$
and
$$F_{n-1}(X_{\delta})X_{T_i}=X_{M(\underline{dim}T_i+(n-1)\delta)}+
X_{M(\underline{dim}T_i-(n-1)\delta)},$$ then we have
$$F_{n+1}(X_{\delta})X_{T_i}=X_{M(\underline{dim}T_i+(n+1)\delta)}+X_{M((n+1)\delta-\underline{dim}T_i)[-1]};$$

When $\underline{dim}T_i<(n-1)\delta$, we obtain
$$X_{\delta}X_{M(n\delta-\underline{dim}T_i)[-1]}=X_{M((n+1)\delta-\underline{dim}T_i)[-1]}
+X_{M((n-1)\delta-\underline{dim}T_i)[-1]},$$
and
$$F_{n-1}(X_{\delta})X_{T_i}=X_{M(\underline{dim}T_i+(n-1)\delta)}+
X_{M((n-1)\delta-\underline{dim}T_i)[-1]},$$ thus we have
$$F_{n+1}(X_{\delta})X_{T_i}=X_{M(\underline{dim}T_i+(n+1)\delta)}+X_{M((n+1)\delta-\underline{dim}T_i)[-1]}.$$
The claim is proved.

Note that $X_{M(\underline{dim}T_i+n\delta)},X_{M(\underline{dim}T_i-n\delta)}$ and $X_{M(n\delta-\underline{dim}T_i)[-1]}$
are all cluster variables, then by \cite{LS}, they belong to $\mathcal{A}(Q)\cap \mathbb{N}[X_{T_1}^{\pm 1},\cdots,X_{T_n}^{\pm
1}]$. Thus, by the above proved claim, we obtain
$$F_{n}(X_{\delta})\in \mathcal{A}(Q)\cap \mathbb{N}[X_{T_1}^{\pm
1},\cdots,X_{T_n}^{\pm 1}].$$
Therefore the result  follows.
\end{proof}

Now we make use of Proposition \ref{Cheby} to prove the main result in this paper.

\begin{Thm}\label{5}
Let $Q$ be a quiver of affine types,  then
$$X_{M}\in \mathcal{A}(Q)\cap
\mathbb{N}[\mathrm{c}^{\pm 1}]$$ for any object $M$ in
$\mathcal{C}(Q)$ and any cluster $\mathrm{c}$.
\end{Thm}
\begin{proof}
For any object $M,N\in \mathcal{C}(Q)$, we have $X_{M\oplus N}=X_{M}X_{N},$ and also it  is well-known that cluster variables are positive,
so we only need to prove the theorem for any indecomposable regular
generalized variables.

Firstly, we consider the case in homogenerous
tubes. Note that
$X_{n\delta}=F_{n}(X_{\delta})+X_{(n-2)\delta}$, thus we can prove
$X_{n\delta}\in \mathbb{N}[\mathrm{c}^{\pm 1}]$ by induction.

Secondly, we consider the case in non-homogenerous tubes. We fix a non-homogenerous tube $\mathcal{T}$ of rank $r$. The quasi-simples of
$\mathcal{T}$ are denoted
by $E_{i}$ with $1\leq i\leq r$ ordered so that $\tau E_{i}=E_{i-1}$.  The regular module with quasi-socle $E$ and quasi-length $k$ for any $k\in \mathbb{N}$ are denoted
by $E[k]$.
According to the
general different property (see \cite[Theorem 3.4]{Dupont2}):
$$X_{E_{i}[nr+k]}=X_{E_{i}[k]}F_{n}(X_{\delta})+X_{E_{i+k+1}[nr-k-2]}$$
where $r$ is the rank of an exceptional tube, $n\geq 0$ and $0\leq k
\leq r-1,$
 and
the positivity of $F_{n}(X_{\delta})$ which is proved in Proposition \ref{Cheby}, we can deduce
$X_{E_{i}[nr+k]}\in \mathcal{A}(Q)\cap
\mathbb{N}[\mathrm{c}^{\pm 1}]$ by induction. Here we only need
to note that $X_{\delta}X_{E_{i}[r-1]}=X_{E_{i}[2r-1]}$ where $0\leq
i \leq r-1.$
\end{proof}

Let $Q$ be a quiver of affine types, then we have the following three integral bases of the cluster algebras $\mathcal{A}(Q)$ (see \cite{DXX,Dupont1,Dupont2}):
$$\mathcal{B}=\mathcal{CM}\cup \{F_{n}(X_{\delta})X_R|R\ \text{is a regular rigid kQ-module and}\ n\geq 1\}$$
$$\mathcal{S}=\mathcal{CM}\cup \{X_{n\delta}X_R|R\ \text{is a regular rigid kQ-module and}\ n\geq 1\}$$
$$\mathcal{G}=\mathcal{CM}\cup \{X_{\delta}^{n}X_R|R\ \text{is a regular rigid kQ-module and}\ n\geq 1\}$$
where  we denote the set of all cluster monomials of the cluster algebras $\mathcal{A}(Q)$ by $\mathcal{CM}$.

 We can now put all these results together  to obtain:
\begin{Cor}\label{basis}
Let $Q$ be a quiver of affine types and $\mathrm{c}$ be any cluster, then we have

$(1)\ \mathcal{B}\in \mathcal{A}(Q)\cap
\mathbb{N}[\mathrm{c}^{\pm 1}];$

$(2)\ \mathcal{S}\in \mathcal{A}(Q)\cap
\mathbb{N}[\mathrm{c}^{\pm 1}];$

$(3)\ \mathcal{G}\in \mathcal{A}(Q)\cap
\mathbb{N}[\mathrm{c}^{\pm 1}];$

\end{Cor}
\begin{proof}
By Proposition \ref{Cheby} and the positivity in cluster variables, we obtain that $$\mathcal{B}\in \mathcal{A}(Q)\cap
\mathbb{N}[\mathrm{c}^{\pm 1}].$$
By Theorem \ref{5} and the positivity in cluster variables, we obtain that
$$\mathcal{S}\in \mathcal{A}(Q)\cap
\mathbb{N}[\mathrm{c}^{\pm 1}]\ \text{and}\ \mathcal{G}\in \mathcal{A}(Q)\cap
\mathbb{N}[\mathrm{c}^{\pm 1}].$$
The result follows.
\end{proof}

\begin{Remark}
(1)\ The  basis $\mathcal{B}$ was initially constructed  for rank 2 cluster algebras of finite and affine types in \cite{SZ} and  for  type $\widetilde{A}_2^{(1)}$ in \cite{GC} where are called canonical bases, and then constructed  for   types $A$ and $\widetilde A$ in \cite{Dupont-h} where are called the atomic basis.

(2)\ The  basis $\mathcal{B}$ was  constructed  for the
Kronecker quiver \cite{CZ} where is called the dual semicanonical basis.

(3)\ The  basis $\mathcal{G}$ was  constructed  in \cite{Dupont1} for type $\widetilde{A}$
and \cite{DXX} for affine types, and  for more general case in \cite{GLSg} where are called the generic basis.
\end{Remark}

\section{An example:  type $\widetilde{D}_{4}$}

We consider  the tame quiver $Q$ of type
$\widetilde{D}_4$ as follows
$$
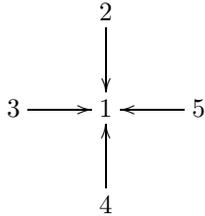
\xymatrix{& 2 \ar[d] &\\
3 \ar[r] & 1  & 5 \ar[l]\\
& 4 \ar[u] &}
$$
 In this case, we will  provide an explicit description of the proof in Proposition \ref{Cheby}.

The category of regular modules decomposes into a direct sum of
tubes indexed by the projective line $\mathbb{P}^1$ among which
there are just three tubes of rank 2 and all other tubes are
homogeneous tubes \cite{DR76}. We denote these three exceptional
tubes labelled by the subset $\{0,1,\infty\}$ of $\mathbb{P}^1$. The
quasi-simple modules in non-homogeneous tubes are denoted by
$E_1,E_2,E_3,E_4,E_5,E_6$, where
$$\underline{\mathrm{dim}}(E_1)=(1,1,1,0,0),\,\,\underline{\mathrm{dim}}(E_2)=(1,0,0,1,1),\,\,
\underline{\mathrm{dim}}(E_3)=(1,1,0,1,0),$$$$
\underline{\mathrm{dim}}(E_4)=(1,0,1,0,1),\,\,
\underline{\mathrm{dim}}(E_5)=(1,0,1,1,0),\,\,\underline{\mathrm{dim}}(E_6)=(1,1,0,0,1).$$
We remark that $\{E_1, E_2\}$, $\{E_3, E_4\}$ and $\{E_5, E_6\}$ are
pairs of the quasi-simple modules at the mouth of exceptional tubes
labelled by $1,\infty$ and $0$, respectively.
 We denote the minimal imaginary root
by $\delta=(2,1,1,1,1)$.
From the Auslander-Reiten quiver of type $\widetilde{D}_4$, we know
that all preprojective modules are the following forms:

1): $C(n)$ with
$\underline{\mathrm{dim}}C(n)=(2n-1,n-1,n-1,n-1,n-1)$, where
$n\geq 1$.

2):  $M_1(n),M_2(n),M_3(n),M_4(n)$ with
$$\underline{\mathrm{dim}}M_1(n)=(n,\frac{n+1}{2},\frac{n-1}{2},\frac{n-1}{2},\frac{n-1}{2}),$$
$$\underline{\mathrm{dim}}M_2(n)=(n,\frac{n-1}{2},\frac{n+1}{2},\frac{n-1}{2},\frac{n-1}{2}),$$
$$\underline{\mathrm{dim}}M_3(n)=(n,\frac{n-1}{2},\frac{n-1}{2},\frac{n+1}{2},\frac{n-1}{2}),$$
$$\underline{\mathrm{dim}}M_4(n)=(n,\frac{n-1}{2},\frac{n-1}{2},\frac{n-1}{2},\frac{n+1}{2}),$$
where $n$ is odd and $n\geq 1.$

3):  $N_1(n),N_2(n),N_3(n),N_4(n)$ with :
$$\underline{\mathrm{dim}}N_1(n)=(n,\frac{n-2}{2},\frac{n}{2},\frac{n}{2},\frac{n}{2}),$$
$$\underline{\mathrm{dim}}N_2(n)=(n,\frac{n}{2},\frac{n-2}{2},\frac{n}{2},\frac{n}{2}),$$
$$\underline{\mathrm{dim}}N_3(n)=(n,\frac{n}{2},\frac{n}{2},\frac{n-2}{2},\frac{n}{2}),$$
$$\underline{\mathrm{dim}}N_4(n)=(n,\frac{n}{2},\frac{n}{2},\frac{n}{2},\frac{n-2}{2}),$$
where $n$ is even and $n\geq 2$. Note that
$$C(1)=P_1,M_1(1)=P_2,M_2(1)=P_3,M_3(1)=P_4,M_4(1)=P_5.$$
The Auslander-Reiten quiver of the preprojective component is as
follows:
$$
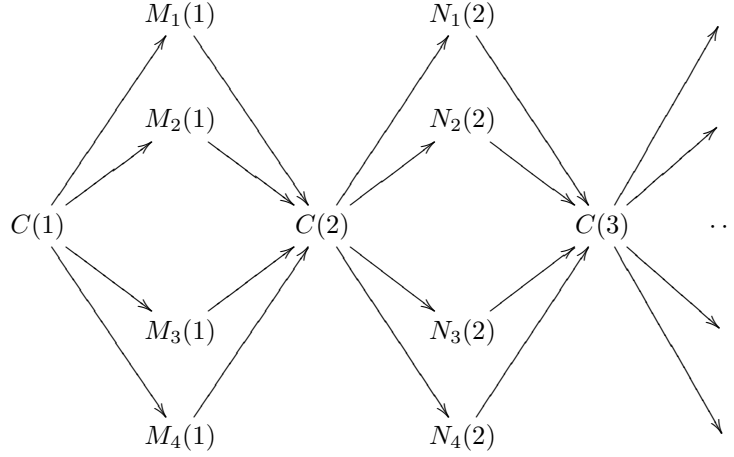
\xymatrix{                                 & M_1(1) \ar[rdd]& &N_1(2) \ar[rdd]& & &\\
                                           & M_2(1) \ar[rd] & &N_2(2) \ar[rd] & & &\\
       C(1)\ar[ruu]\ar[ru]\ar[rd]\ar[rdd] &                 &C(2) \ar[ruu]\ar[ru]\ar[rd]\ar[rdd]& &C(3) \ar[ruu]\ar[ru]\ar[rd]\ar[rdd]& \cdots \\
       & M_3(1) \ar[ru]& &N_3(2) \ar[ru]& & &\\
       & M_4(1) \ar[ruu]& &N_4(2) \ar[ruu]& &&}
$$

All preinjective modules are the following forms:

1): $C'(n)$ with $\underline{\mathrm{dim}}C'(n)=(2n-1,n,n,n,n),$
where $n\geq 1$.

2):  $M'_1(n),M'_2(n),M'_3(n),M'_4(n)$ with When
$$\underline{\mathrm{dim}}M'_1(n)=(n-1,\frac{n+1}{2},\frac{n-1}{2},\frac{n-1}{2},\frac{n-1}{2}),$$
$$\underline{\mathrm{dim}}M'_2(n)=(n-1,\frac{n-1}{2},\frac{n+1}{2},\frac{n-1}{2},\frac{n-1}{2}),$$
$$\underline{\mathrm{dim}}M'_3(n)=(n-1,\frac{n-1}{2},\frac{n-1}{2},\frac{n+1}{2},\frac{n-1}{2}),$$
$$\underline{\mathrm{dim}}M'_4(n)=(n-1,\frac{n-1}{2},\frac{n-1}{2},\frac{n-1}{2},\frac{n+1}{2}),$$
where $n$ is odd and $n\geq 1$.

3):  $N'_1(n),N'_2(n),N'_3(n),N'_4(n)$ with
$$\underline{\mathrm{dim}}N'_1(n)=(n-1,\frac{n-2}{2},\frac{n}{2},\frac{n}{2},\frac{n}{2}),$$
$$\underline{\mathrm{dim}}N'_2(n)=(n-1,\frac{n}{2},\frac{n-2}{2},\frac{n}{2},\frac{n}{2}),$$
$$\underline{\mathrm{dim}}N'_3(n)=(n-1,\frac{n}{2},\frac{n}{2},\frac{n-2}{2},\frac{n}{2}),$$
$$\underline{\mathrm{dim}}N'_4(n)=(n-1,\frac{n}{2},\frac{n}{2},\frac{n}{2},\frac{n-2}{2}),$$
where $n$ is even and $n\geq 2$.  Note that
$$C'(1)=I_1,M'_1(1)=I_2,M'_2(1)=I_3,M'_3(1)=I_4,M'_4(1)=I_5.$$
The Auslander-Reiten quiver of the preinjective component is as
follows:
$$
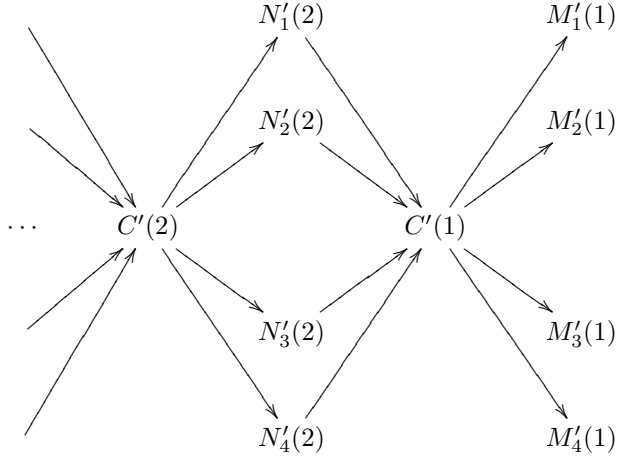
\xymatrix{  \ar[rdd]&                               & N'_1(2) \ar[rdd]& &M'_1(1) \\
           \ar[rd]  &                              & N'_2(2) \ar[rd] & &M'_2(1) \\
  \cdots    & C'(2)\ar[ruu]\ar[ru]\ar[rd]\ar[rdd] &                 &C'(1) \ar[ruu]\ar[ru]\ar[rd]\ar[rdd]\\
  \ar[ru] &    & N'_3(2) \ar[ru]& &M'_3(1) \\
   \ar[ruu] &   & N'_4(2) \ar[ruu]& &M'_4(1)}
$$
The following result  proved in \cite{DX}  is  useful for us to get the positivity.
\begin{Prop}\cite{DX}\label{D4}
Assume that $n\geq 1$, then we have

(1) If $n=1$, then
$$\ X_{\delta}X_{P_2}=X_{M_1(3)}+X_{M'_{1}(1)};$$

(2) If $n\geq 3$ is odd, then
$$\ X_{\delta}X_{M_1(n)}=X_{M_1(n+2)}+X_{M_1(n-2)}.$$
\end{Prop}
 Similar to Proposition \ref{D4}(2), it is easy to show  the following result directly.

 \begin{Prop}\label{D4+}
If $n\geq 3$ is odd, then
$$\ X_{\delta}X_{M'_1(n)}=X_{M'_1(n+2)}+X_{M'_1(n-2)}.$$
\end{Prop}

For any  cluster-tilting object $T$, it is easy to see that   there exists at least  a direct summand $T_i$ of $T$ such that
$\tau^{m}T_i$ is equal to some $P_{k}$ for $2\leq k\leq 5$ and some certain $m\in\mathbb{Z}$. Note that  $\tau$
induces an equivalence of cluster categories and maps
cluster-tilting object to another cluster-tilting object.
Without loss of generality, we can assume that $T_i=P_2$, then we have the following result.
\begin{Prop}\label{D4cheby}
Assume that $n\geq 1$, then we have $$F_{n}(X_{\delta})X_{P_2}=X_{M_1(2n+1)}+X_{M'_{1}(2n-1)}.$$
\end{Prop}
\begin{proof}
We  prove it by induction.
When $n=1,$ it follows from the Proposition \ref{D4}(1). When $n=2$, note that $P_2=M_{1}(1)$, we have
\begin{eqnarray}
\nonumber F_{2}(X_{\delta})X_{P_2} & = & (X^{2}_{\delta}-2)X_{P_2}\\
\nonumber  & = & X_{\delta}(X_{M_1(3)}+X_{M'_{1}(1)})-2X_{P_2}\\
\nonumber & = & X_{M_1(5)}+X_{M_{1}(1)}+X_{M'_1(3)}+X_{M_{1}(1)}-2X_{P_2}\\
\nonumber & = & X_{M_1(5)}+X_{M'_1(3)}.
\end{eqnarray}
Now suppose that it holds for $n\leq k$. When $n=k+1$, then by Proposition \ref{D4} and Proposition \ref{D4+},
we have
\begin{eqnarray}
\nonumber F_{k+1}(X_{\delta})X_{P_2} & = & (X_{\delta}F_{k}(X_{\delta})-F_{k-1}(X_{\delta}))X_{P_2}\\
\nonumber  & = & X_{\delta}(X_{M_1(2k+1)}+X_{M'_{1}(2k-1)})-(X_{M_1(2k-1)}+X_{M'_{1}(2k-3)})\\
\nonumber & = & X_{M_1(2k+3)}+X_{M_{1}(2k-1)}+X_{M'_1(2k+1)}+X_{M'_{1}(2k-3)}\\
\nonumber & - &(X_{M_1(2k-1)}+X_{M'_{1}(2k-3)})\\
\nonumber & = & X_{M_1(2k+3)}+X_{M'_1(2k+1)}.
\end{eqnarray}
Thus the proof is finished.
\end{proof}

It is obvious to see that $X_{M_1(2n+1)}$ and $X_{M'_{1}(2n-1)}$
are all cluster variables, so both of them belong to $\mathcal{A}(Q)\cap \mathbb{N}[X_{T_1}^{\pm 1},\cdots,X_{T_n}^{\pm
1}]$. Note  that $X_{P_2}=X_{T_i}$, then by Proposition \ref{D4cheby}, we get
$$F_{n}(X_{\delta})\in \mathcal{A}(Q)\cap \mathbb{N}[X_{T_1}^{\pm
1},\cdots,X_{T_n}^{\pm 1}].$$

\end{document}